\documentclass{amsart}

\usepackage{amsmath,amssymb,amsthm,pinlabel}

\def\co{\colon\thinspace}
\newcommand{\Int}{\mbox{\rm Int}}
\newcommand{\e}{\mathrm{e}}
\newcommand{\U}{\mathrm{U}}

\newcommand{\SO}{\mathrm{SO}}

\newcommand{\N}{\mathbb{N}}
\newcommand{\R}{\mathbb{R}}
\newcommand{\Z}{\mathbb{Z}}

\newcommand{\oz}{\overline{z}}

\newcommand{\oW}{\overline{W}}
\newcommand{\ttt}{{\tt t}}
\newcommand{\tb}{{\tt tb}}

\newcommand{\rot}{{\tt rot}}
\newcommand{\cpk}{{\mathbb {CP}}^2}
\newcommand{\cpl}{{\mathbb {CP}}^1}
\newcommand{\rpk}{{\mathbb {RP}}^3}

\newtheorem{thm}{Theorem}
\newtheorem{lem}[thm]{Lemma}
\newtheorem{prop}[thm]{Proposition}
\newtheorem{cor}[thm]{Corollary}

\theoremstyle{definition}
\newtheorem*{rem}{Remark}
\newtheorem*{rems}{Remarks}

\newtheorem*{ex}{Example}
\newtheorem*{ack}{Acknowledgements}

\begin{document}

\title[Contact structures on products and fibre sums]{Contact structures
on product $5$-manifolds\\and fibre sums along circles}

\author{Hansj\"org Geiges}
\address{Mathematisches Institut, Universit\"at zu K\"oln,
Weyertal 86--90, 50931 K\"oln, Germany}
\email{geiges@math.uni-koeln.de}
\author{Andr\'as I. Stipsicz}
\address{A. R\'enyi Institute of Mathematics, Hungarian Academy of Sciences,
Re\'al\-ta\-noda utca 13--15, 1053 Budapest, Hungary}
\email{stipsicz@renyi.hu}
\date{}

\begin{abstract}
Two constructions of contact manifolds are presented: (i) pro\-ducts
of $S^1$ with manifolds admitting a suitable decomposition into two
exact symplectic pieces and (ii) fibre connected sums
along isotropic circles. Baykur has found a decomposition as required for
(i) for all closed, oriented $4$-manifolds.
As a corollary, we can show that all closed, oriented $5$-manifolds
that are Cartesian products of lower-dimensional manifolds carry
a contact structure. For symplectic $4$-manifolds we exhibit
an alternative construction of such a decomposition; this gives
us control over the homotopy type of the corresponding contact structure.
In particular, we prove that $\cpk\times S^1$ admits a contact structure
in every homotopy class of almost contact structures.
The existence of contact structures is also
established for a large class of $5$-manifolds with fundamental group~$\Z_2$.
\end{abstract}

\maketitle

\section{Introduction}
Contact structures, by virtue of their maximal non-integrability,
appear to be better adapted to twisted products (i.e.\
fibre bundles) than to Cartesian products of manifolds.
Classical examples supporting this dictum are the natural contact structures
on unit cotangent bundles $ST^*B$ (which are Cartesian products only if
the base manifold $B$ is parallelisable) and the Boothby--Wang~\cite{bowa58}
construction, where the contact form is the connection $1$-form
on an $S^1$-bundle, with curvature form equal to a symplectic form
on the base representing the Euler class of the bundle.
More recently, Lerman~\cite{lerm04} has shown how to build contact
structures on bundles whose curvature satisfies a certain
non-degeneracy condition. Lerman's work is a contact version of Sternberg's
minimal coupling construction and Weinstein's fat bundles in
symplectic geometry, and it generalised earlier results of Yamazaki.

By contrast, little is known about contact structures on Cartesian
products. Some simple examples, such as products of spheres,
can be obtained by contact surgery on a sphere. Manifolds
of the form $M^3\times S^2$ admit a contact structure thanks to
the parallelisability of (closed, oriented) $3$-manifolds.
Using a branched covering construction, one can then also put
a contact structure on the product of $M^3$ with an arbitrary
(closed, oriented) surface, see~\cite{geig97}.

But it took more than twenty years from Lutz's discovery of
a contact structure on the $5$-torus to a construction by
Bourgeois~\cite{bour02} of a contact structure on higher-dimensional
tori and, more generally, products of contact manifolds with
a surface of genus at least one.

In Section~\ref{section:product} of the present note we exhibit a method
for putting contact
structures on Cartesian products of even-dimensional manifolds~$V$
with~$S^1$, provided the manifold $V$ admits a suitable
decomposition into two exact symplectic pieces (Theorem~\ref{thm:product}).
For closed, oriented $4$-manifolds, such a decomposition has been
shown to exist by Baykur~\cite{bayk06}. As a corollary,
we find that any closed, oriented $5$-manifold that is the Cartesian product
of lower-dimensional manifolds admits a contact structure
(Corollary~\ref{cor:five}).

In Section~\ref{section:decomp} we prove an analogue of
Baykur's decomposition result for the special case of
closed symplectic $4$-manifolds (Theorem~\ref{thm:decomp}).
Our proof rests on Donaldson's theorem
about codimension~$2$ symplectic submanifolds, as well as the classification
of tight contact structures on circle bundles over surfaces
due to Giroux and Honda. (By contrast, Baykur's argument is based on
the theory of Lefschetz fibrations
on $4$-manifolds and open book decompositions of $3$-manifolds
adapted to contact structures.)
Signs are important for the discussion in Section~\ref{section:decomp},
so in Section~\ref{section:BW} we have inserted a brief description
of Boothby--Wang contact forms and their convex resp.\ concave fillings.

In Section~\ref{section:examples} we discuss some simple examples,
notably $V=\cpk$. Here we need not invoke Donaldson's theorem.
Furthermore, the explicit nature of our decomposition theorem
allows us to show, for instance, that $\cpk\times S^1$ admits
a contact structure in every homotopy class of almost contact
structures (Proposition~\ref{prop:cp2}).

A further construction of contact manifolds will be presented
in Section~\ref{section:sum}: the connected sum along isotropic circles
(Theorem~\ref{thm:sum}),
which may be regarded as a counterpart to the fibre connected sum along
codimension~2 contact submanifolds described in~\cite{geig97},
cf.~\cite[Chapter~7.4]{geig08}.

Finally, in Section~\ref{section:Z2} we use the connected sum along
circles to produce contact structures on a class of $5$-manifolds
with fundamental group~$\Z_2$ described recently
by Hambleton and Su~\cite{hasu09}
(Proposition~\ref{prop:five}). This class is in some respects more
restricted than that discussed by Charles Thomas and the first
author~\cite{geth98}, but the Hambleton--Su description is more explicit,
and it includes some additional cases. In fact, it was their classification
result that prompted us to consider
the questions discussed in the present note.

We close this introduction with a comment on notational and other
conventions. We write $M,N$ for closed odd-dimensional (typically:
contact) manifolds. Our contact structures are always assumed to be
cooriented, i.e.\ defined as $\ker\alpha$ with $\alpha$ a globally
defined $1$-form, unique up to multiplication by a positive function.
A $(2n+1)$-dimensional contact manifold $(N,\ker\alpha )$ is given
the orientation induced by the volume form $\alpha\wedge (d\alpha )^n$.

The notation $V$ stands for a closed even-dimensional
(typically: symplectic) manifold; $W$ is used for compact
even-dimensional manifolds with boundary (typically: a symplectic filling
of a contact manifold).
Given a $2n$-dimensional symplectic manifold $(W,\omega )$, we equip it with
the orientation induced by the volume form~$\omega^n$. We write
$\oW$ for the manifold with the opposite orientation.

Unless stated otherwise, (co-)homology is understood with integer
coefficients.
\section{Contact structures on products with $S^1$}
\label{section:product}
A compact manifold $W$ with an exact symplectic form $\omega =d\lambda$
is called an {\em exact symplectic filling} of the contact manifold
$(M,\xi )$ if the following conditions hold:
\begin{itemize}
\item $M=\partial W$ as oriented manifolds,
\item the pull-back of $\lambda$ to $M$ under the inclusion map
is a contact form defining the contact structure~$\xi$.
\end{itemize}
These conditions imply that
the Liouville vector field $X$ on $(W,\omega )$ defined
by $i_X\omega =\lambda$ is pointing outwards along the boundary~$M$,
so any exact filling is in particular a strong filling.

\begin{rem}
Every Stein filling is an exact filling.
With a construction described in~\cite{elia91},
any (weak or strong) symplectic filling
with the property that the symplectic form is exact
can be modified in a collar neighbourhood of the boundary so as to become
an exact filling (of the same contact manifold);
see also~\cite{geig06}. For more information about the
various notions of filling see \cite[Chapter~5]{geig08} and
\cite[Chapter~12]{ozst04}.
\end{rem}

We can use the flow of $X$ to define a collar neighbourhood
$(-\varepsilon ,0]\times M$ of $M$ in~$W$. Here $\{0\}\times M$ is
identified with $M=\partial W$, and with $t$ denoting the parameter
in $(-\varepsilon ,0]$ we have $\partial_t=X$. Then, on this collar
neighbourhood, the symplectic form $\omega$ can be written as
$\omega= d(\e^t\lambda )$.

This allows us to define a symplectic completion $W\cup_M \bigl(
[0,\infty )\times M\bigr)$, where the symplectic form on
$W$ is $\omega$, and on $[0,\infty )\times M$ it is $d(\e^t\lambda )$.

Given any non-negative function $h$ on~$M$, we have an embedding
\[
\begin{array}{rclcr}
M & \longrightarrow & [0,\infty ) & \times & M\\
x & \longmapsto     & (h(x)       & ,      & x)
\end{array}
\]
under which the $1$-form $\e^t\lambda$ pulls back to $\e^h\lambda$.
Moreover, two contact forms defining the same (cooriented) contact
structure on a compact manifold $M$ can be made to coincide after multiplying
each of them with a suitable function $f_i\co M\rightarrow [1,\infty )$,
$i=1,2$. Thus, given two exact symplectic fillings $(W_i,d\lambda_i)$,
$i=1,2$, of the same contact manifold $(M,\xi )$, we may assume without loss
of generality that $\lambda_1|_{TM}=\lambda_2|_{TM}$.

\begin{thm}
\label{thm:product}
Let $(W_1,d\lambda_1)$ and $(W_2,d\lambda_2)$ be two exact symplectic
fillings of the same contact manifold $(M,\xi )$. Then the manifold
$(W_1\cup_M\oW_2)\times S^1$ admits a contact structure.
\end{thm}

\begin{proof}
As explained, we may assume that $\lambda_1|_{TM}=\lambda_2|_{TM}$.
Write $\beta$ for this contact form on~$M$. We then find collar neighbourhoods
$(-1-\varepsilon ,-1]\times M$ of $\{ -1\}\times M\equiv M=\partial W_i$
in $W_i$ where $\lambda_i=\e^{t+1}\beta$, $i=1,2$. (This shift in the
collar parameter by $1$ is made for notational convenience in the
construction that follows.)

Now choose two smooth functions $f$ and $g$ on the interval
$(-1-\varepsilon , 1+\varepsilon )$ subject to the following conditions
(see Figure~\ref{figure:functions}):
\begin{itemize}
\item $f$ is an even function with $f(t)=\e^{t+1}$ near $(-1-\varepsilon ,-1]$,
\item $g$ is an odd function with $g(t)=1$ near $(-1-\varepsilon ,-1]$,
\item $f'g-fg'>0$.
\end{itemize}

\begin{figure}[h]
\labellist
\small\hair 2pt
\pinlabel $t$ [t] at 246 145
\pinlabel $t$ [t] at 605 145
\pinlabel $g(t)$ [r] at 486 282
\pinlabel $f(t)$ [r] at 126 282
\pinlabel $1$ [r] at 117 217
\pinlabel $1$ [r] at 477 217
\pinlabel $-1$ [t] at 54 135
\pinlabel $1$ [t] at 198 135
\pinlabel $-1$ [t] at 414 135
\pinlabel $1$ [t] at 558 135
\endlabellist
\centering
\includegraphics[scale=0.5]{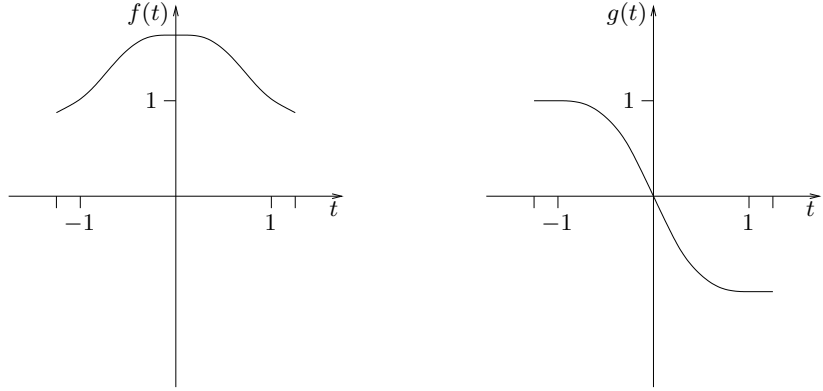}
  \caption{The functions $f$ and $g$.}
  \label{figure:functions}
\end{figure}

Now we consider the manifold
\[  N:=\bigl( W_1\cup_M ([-1,1]\times M)\cup_M\oW_2\bigr)\times S^1,\]
which is a diffeomorphic copy of the manifold in the theorem.
We write $\theta$ for the $S^1$-coordinate.
Define a smooth $1$-form $\alpha$ on $N$ by
\[ \alpha =\left\{
\begin{array}{ll}
\lambda_1+d\theta & \mbox{\rm on}\;\; \Int (W_1)\times S^1,\\[.2mm]
f\beta +g\, d\theta & \mbox{\rm on}\;\; (-1-\varepsilon ,1+\varepsilon )
           \times M\times S^1,\\[.2mm]
\lambda_2-d\theta & \mbox{\rm on}\;\; \Int (\oW_2)\times S^1.
\end{array}\right. \]
Then $\alpha\wedge (d\alpha )^n$ (where $\dim N =2n+1$) equals
\[ \left\{
\begin{array}{ll}
(d\lambda_1)^n\wedge d\theta & \mbox{\rm on}\;\; \Int (W_1)\times S^1,\\[.2mm]
nf^{n-1}(f'g-fg')\, dt\wedge \beta\wedge (d\beta )^{n-1}\wedge d\theta &
 \mbox{\rm on}\;\; (-1-\varepsilon ,1+\varepsilon ) \times M\times S^1,\\[.2mm]
-(d\lambda_2)^n\wedge d\theta & \mbox{\rm on}\;\; \Int (\oW_2)\times S^1.
\end{array}\right. \]
This shows that $\alpha$ is a contact form on~$N$.
\end{proof}

\begin{rem}
We discovered this construction by a slightly roundabout route.
Given an exact symplectic filling $(W,d\lambda)$, one can form
an open book decomposition of a manifold $N_W$ with page $W$ and
monodromy the identity. As
Giroux~\cite{giro02} has observed --- generalising a 3-dimensional
construction due to Thurston and Winkelnkemper --- this manifold $N_W$
carries an adapted contact form~$\alpha$, cf.~\cite[Theorem~7.3.3]{geig08}.
Starting from such a contact form adapted to an open book,
Bourgeois~\cite{bour02} constructed a $T^2$-invariant
contact form on the product of $N_W$
with a $2$-torus~$T^2$, cf.~\cite[Theorem~7.3.6]{geig08}. The contact
reduction (cf.~\cite[Chapter~7.7]{geig08}) with respect to the $S^1$-action
of one of the factors in $T^2=S^1\times S^1$ yields a contact structure on
$(W\cup_{\partial W}\oW)\times S^1$.
The contact structure near $\partial W\times S^1$ is independent
of the topology of the exact symplectic filling, so this construction
generalises as described in the preceding proof.
\end{rem}

\begin{ex}
If one takes both $W_i$ equal to the $2$-disc, and $\lambda_i=x\, dy-y\, dx$,
the above construction yields, up to isotopy, the standard tight contact
structure $\ker (z\, d\theta+x\, dy -y\, dx)$ on $S^2\times S^1$.
\end{ex}

As shown by Hirzebruch and Hopf~\cite{hiho58}, for any closed,
oriented $4$-manifold $V$ the third integral
Stiefel--Whitney class $W_3(V)$ vanishes. Then we also have
$W_3(V\times S^1)=0$, which implies that $V\times S^1$ admits an almost
contact structure~\cite[Proposition~8.1.1]{geig08}. The same is true for
$5$-manifolds of the form $M^3\times\Sigma^2$, since all closed, oriented
$3$-manifolds are spin.

\begin{cor}
\label{cor:five}
Any closed, oriented $5$-manifold that is the Cartesian product of
two lower-dimensional manifolds admits a contact structure.
\end{cor}

\begin{proof}
For manifolds of the form $M^3\times\Sigma^2$ this was
proved in~\cite{geig97}.
For manifolds of the form $V^4\times S^1$ we appeal to a result
of Baykur~\cite{bayk06}, according to which  any closed, oriented
$4$-manifold $V$ admits a decomposition $V=W_1\cup_M\oW_2$ as required
by Theorem~\ref{thm:product}. In fact, $W_1$ and $W_2$ may be taken to be
{\em Stein} fillings of the same contact manifold $(M,\xi )$.
\end{proof}

It was first shown by Akbulut and Matveyev~\cite{akma98}
that any closed, oriented $4$-manifold can be decomposed
into two Stein pieces, but their result did not give any information
about the induced contact structures on the separating hypersurface.
Baykur's result provides this additional information by an approach via
the theory of Lefschetz fibrations
on $4$-manifolds and open book decompositions of $3$-manifolds
adapted to contact structures. In Section~\ref{section:decomp}
we prove a similar decomposition result for the subclass of {\em symplectic}
$4$-manifolds, based on Donaldson's theorem
about codimension~$2$ symplectic submanifolds and the classification
of tight contact structures on circle bundles over surfaces
due to Giroux and Honda. The following section serves
as a preparation for Section~\ref{section:decomp}.
\section{Convex vs.\ concave fillings of Boothby--Wang contact forms}
\label{section:BW}
We briefly recall some facts about the Boothby--Wang construction;
for more details see~\cite[Chapter~7.2]{geig08}.

Let $\pi\co M\rightarrow B$ be a principal $S^1$-bundle over a closed,
oriented manifold~$B$.
By a connection $1$-form on $M$ we mean a differential $1$-form
$\alpha$ which is invariant ($\mathcal{L}_{\partial_{\theta}}\alpha\equiv 0$)
and normalised ($\alpha(\partial_{\theta})\equiv 1$). Note that the
relation with the usual $i\R$-valued connection $1$-form $A$ is
$\alpha =-iA$. The differential $d\alpha$ induces a well-defined
closed $2$-form $\omega$ on~$B$, i.e.\ $d\alpha =\pi^*\omega$. This
$\omega$ is called the curvature form of the connection $1$-form~$\alpha$.
The Euler class of the bundle is given by $e=-[\omega /2\pi ]$.

We also write $\pi\co W\rightarrow B$ for the projection map in
the associated $D^2$-bundle. With $r$ denoting the radial coordinate on~$D^2$,
we can extend $\alpha$ to an $r$-invariant $1$-form on $W$ outside
the zero section; the $1$-form $r^2\alpha$ is well defined and
smooth on all of~$W$.

We say that the $S^1$- or $D^2$-bundle of Euler class $e$ is {\em positive\/}
(resp.\ {\em negative\/}) if the base manifold $B$ admits a symplectic
form $\omega$ (compatible with the orientation of~$B$) such that
the cohomology class $[\omega /2\pi ]$ equals~$e$ (resp.~$-e$).
Clearly, in either case the connection $1$-form $\alpha$ (on the
$S^1$-bundle with Euler class~$e$)
with curvature form $\mp\omega$ --- such an $\alpha$ can always be found ---
will be a contact form on~$M$. (Notice that for $\dim B =4m$,
a $2$-form $\omega$ on $B$ is symplectic if and only if $-\omega$
is symplectic, so any positive bundle will also be negative,
and vice versa.)

The concave part of the following lemma is~\cite[Lemma~2.6]{mcdu91};
the convex part is completely analogous. This has also been observed by
Niederkr\"uger~\cite{nied05}.

\begin{lem}
\label{lem:con}
If $e$ is positive (resp.\ negative), then there is a symplectic form
$\Omega$ on $W$ which makes it a strong concave (resp.\ convex) filling
of $(M,\ker(\mp\alpha ))$.
\end{lem}

\begin{proof}
Depending on whether $e$ is positive or negative, we have
$\pm e=[\omega /2\pi]$ for some symplectic form $\omega$ on~$B$.
We find a connection $1$-form $\alpha$ with $d\alpha=\mp\pi^*\omega$.
In both cases we set
\[ \Omega =\frac{1}{2}d(r^2\alpha )+\pi^*\omega .\]
This is a symplectic form on $W$, compatible with its orientation,
and a Liouville vector field defined near
$\partial W$ (in fact, everywhere outside the zero section)
and pointing into (resp.\ out of) $W$ along the boundary is given
by $X=\bigl( (r^2\mp 2)/2r\bigr)\,\partial_r$.
This completes the proof.
\end{proof}
\section{Decompositions of symplectic four-manifolds}
\label{section:decomp}
The following is a special case of Baykur's result
(and with exact symplectic rather than Stein fillings),
but it will be proved by entirely different methods.
As we shall see in Section~\ref{section:examples},
the decomposition we construct of a symplectic $4$-manifold $V$ is explicit
enough, in specific cases, to allow control over the homotopy type of the
corresponding contact structure on $V\times S^1$.

\begin{thm}
\label{thm:decomp}
Suppose that $(V,\omega )$ is a closed symplectic $4$-manifold. Then there are
exact symplectic fillings $W_1, W_2$ of a contact $3$-manifold $(M, \xi)$
with the property that $W_1\cup _M \oW_2$ is diffeomorphic to $V$,
where we glue using a contactomorphism of the boundaries $\partial W_1$
and $\partial W_2$.
\end{thm}

Before turning to the proof of this theorem, we collect a few
observations. By (possibly) perturbing and rescaling the
symplectic form $\omega$ we can assume that $\omega /2\pi$
represents an integral
cohomology class in $H^2 (V; \R )$. According to a famous result of
Donaldson~\cite{dona96}, for $k\in\N$ sufficiently large the Poincar\'e
dual of the cohomology class $k[\omega /2\pi ]$ can be represented by a
$2$-dimensional connected symplectic submanifold $\Sigma \subset V$.

\begin{prop}
\label{prop:exact}
Suppose that the connected symplectic surface $\Sigma$ in the
symplectic $4$-manifold $(V, \omega )$
represents the Poincar\'e dual of $k[\omega /2\pi ]$ in the second homology
group $H_2(V)$. Then there is a closed tubular neighbourhood $\nu\Sigma$
of $\Sigma$ with $\omega$-concave boundary $\partial (\nu\Sigma )$.
Write $M$ for this boundary, oriented as the boundary of the closure $W_1$ of
the complement of~$\nu\Sigma$. With $\xi$ denoting the induced
contact structure on~$M$, the symplectic manifold $(W_1,\omega )$
is --- after a suitable modification of the symplectic form in
a neighbourhood of $M$ in~$W_1$ --- an exact filling of the contact
$3$-manifold $(M ,\xi )$.
\end{prop}

\begin{proof}
Since the self-intersection $[\Sigma ]^2=k^2 \int_V\omega \wedge
\omega /4\pi^2$ is positive, the stipulated $\omega$-concave neighbourhood
$\nu \Sigma$ exists by Lemma~\ref{lem:con} and the symplectic
neighbourhood theorem~\cite[Theorem~3.30]{mcsa98},
cf.~\cite[Corollary~6]{etny98}.

The class $k[\omega /2\pi ]$, being Poincar\'e dual to the homology class
represented by~$\Sigma$, may be represented by
a differential form with support in the interior of~$\nu\Sigma$
(this is known as the localisation principle,
see~\cite{botu82}). It follows that, under the inclusion $W_1\subset V$,
the class $[\omega /2\pi ]$ pulls back to the zero class in $H^2(W_1;\R )$.
This means that $\omega|_{W_1} =d\beta$ for some $1$-form $\beta$ on~$W_1$.

A priori, $(W_1,d\beta )$ is not an exact filling of $(M,\xi)$, because
the Liouville vector field $Y$ defined by $i_Y\omega =\beta$ need not
be transverse to~$M$. However, as pointed out in the remark at
the beginning of Section~\ref{section:product}, we can modify the
symplectic form in a collar neighbourhood of the boundary
so as to obtain an exact filling.
\end{proof}

For the proof of the main result of this section we need a more
precise understanding of the positive (and fillable,
hence tight) contact structure $\xi$ on the
$3$-manifold~$M$. Notice that with our orientation convention
for $M$ (see Proposition~\ref{prop:exact}), this manifold
is an $S^1$-bundle over $\Sigma$ with negative Euler number~$e=-[\Sigma ]^2$.

Write $g$ for the genus of~$\Sigma$.
The adjunction equality $2g-2=[\Sigma ]^2-\langle c_1(V),[\Sigma ]\rangle$,
cf.~\cite[Theorem~3.1.9]{ozst04},
shows that by choosing $k$ (in the above proposition) sufficiently large,
which Donaldson's theorem allows us to do, we may always assume that
$g>1$.

Recall that tight contact structures on $S^1$-bundles admit a
numerical invariant, called the {\em maximal twisting\/} $\ttt$, which measures
the maximal contact framing of a Legendrian knot smoothly isotopic to
the fibre, when measured with respect to the framing the fibre
inherits from the fibration. If the contact framing is non-negative
with respect to the fibration framing, traditionally one declares
$\ttt =0$. (For more on twisting see~\cite{hond00b}.)

According to the proof of Lemma~\ref{lem:con},
the contact structure $\xi $ on $M$ is
{\em horizontal}, that is, the contact
planes are transverse to the fibres of the $S^1$-fibration. This
property implies that $\xi$ has negative maximal twisting,
see the proof of~\cite[Theorem~3.8]{hond00b}. On the
other hand, according to the classification of tight contact
structures on a circle bundle with negative Euler number
over a surface of genus $g>1$, there are exactly two contact structures
(up to isotopy) which are horizontal~\cite[Theorem~2.11]{hond00b},
and both are universally tight~\cite[Lemma~3.9]{hond00b}.

\begin{figure}[h]
\centering
\includegraphics[scale=0.2]{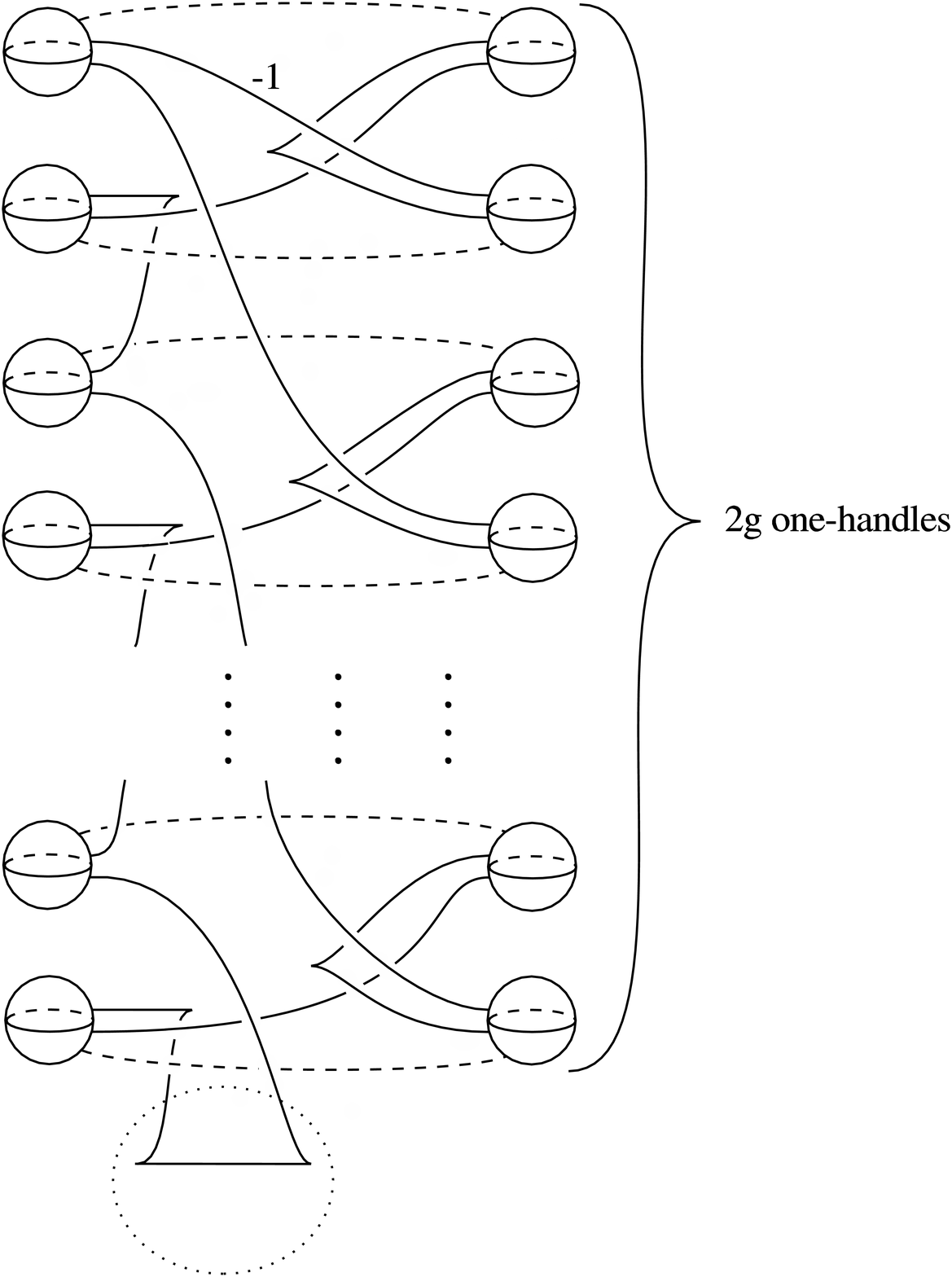}
  \caption{The surgery diagram for $(M,\xi )$.}
  \label{figure:stein}
\end{figure}

We claim that these two contact structures can be described by the
contact surgery diagram of Figure~\ref{figure:stein}
(which is taken from~\cite{list04}) after the surgery curve
has been stabilised in one direction $2g-2-e$ times, i.e.,
after the dotted ellipse in Figure~\ref{figure:stein} has been
replaced by the diagram in Figure~\ref{figure:stabilise},
which contains $2g-2-e$ additional zigzags on the left,
or by the mirror image of that diagram.
Indeed, the Legendrian knot $K$ shown in the diagram has Thurston--Bennequin
invariant $\tb (K)=2g-1$ (and rotation number
$\rot (K)=0$), so contact $(-1)$-surgery on it
produces the circle bundle over the surface of genus $g$ with
Euler number $2g-2$, cf.~\cite[Example~11.2.4]{gost99}.
Stabilising $K$ before the surgery has the effect of reducing the Euler
number. It turns out that each of the $2g-1-e$ tight contact structures
on $M$ with negative twisting can be obtained via one of the $2g-1-e$
different ways of performing $2g-2-e$
stabilisations~\cite[Section~3.2]{hond00b}. The result will be universally
tight, however, only in case the stabilisations are all of the same sign,
see~\cite[p.~435]{gost99}.

\begin{figure}[h]
\centering
\includegraphics[scale=0.25]{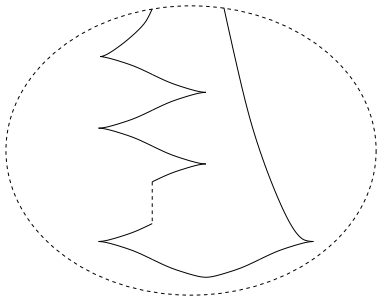}
  \caption{The stabilisations.}
  \label{figure:stabilise}
\end{figure}

Notice that, since contact $(-1)$-surgery corresponds to attaching
a Stein $2$-handle, Figure~\ref{figure:stein} (with the
appropriate stabilisations of~$K$) actually describes a Stein
filling $W_2$ of $(M,\xi )$ by the disc bundle over $\Sigma$ with Euler
number~$e$.

\begin{rems}
(1) Lemma~\ref{lem:con} also provides a strong symplectic filling of
$(M,\xi )$ with the topology of~$W_2$, but this contains the
closed surface $\Sigma$ (i.e.\ the zero section) as
a symplectic submanifold, so this filling is far from being exact.

(2) An alternative argument for the existence of a Stein structure on~$W_2$,
which does not rely on the classification of contact structures,
can be given by appealing to a result of Bogomolov and
de Oliveira~\cite{bool97}. Consider the holomorphic $\cpl$-bundle
over $\Sigma$ associated to the $S^1$-bundle of Euler number~$-e>0$.
This has a $0$- and an $\infty$-section. The complement of
a neighbourhood of the $0$-section gives a holomorphic filling
of $(M,\xi )$ with the desired topology of~$W_2$. The main result
from \cite{bool97} then tells us that a small deformation of
that holomorphic filling will be a Stein filling (or potentially the
blow-up of a Stein filling, but this can only occur for $g=0$
and~$e=-1$.)

(3) For symplectic manifolds $(V,\omega )$ of higher dimensions $2n>4$,
the manifold $W_2$ will be a $2$-disc bundle
over a $(2n-2)$-dimensional symplectic submanifold $\Sigma\subset V$,
so it has homology in dimension $2n-2>n$ and cannot possibly carry
a Stein structure. Note, however, that in order to extend our
result to higher dimensions it would be enough
to find an exact symplectic form (filling the boundary contact
structure). As first observed by McDuff~\cite{mcdu91}
and discussed further in~\cite[Section~7]{geig97a}, the homological
restrictions for Stein fillings do not apply to exact symplectic fillings.
\end{rems}

\begin{proof}[Proof of Theorem~\ref{thm:decomp}]
Simply take $W_1$ as in Proposition~\ref{prop:exact}, and
$W_2$ as in the preceding discussion.
\end{proof}
\section{Examples}
\label{section:examples}
Here are a couple of simple cases where we do not need to invoke Donaldson's
theorem for proving Theorem~\ref{thm:decomp}. In the second example we
also show how to obtain additional information about the homotopy
classification of the contact structures obtained
via Theorem~\ref{thm:product}.

\vspace{1mm}

(1) If the symplectic manifold $V$ we want to decompose
as $W_1\cup\oW_2$ is a ruled surface, that is, if it admits
an $S^2$-bundle structure over some surface~$\Sigma$, then we can take
$W_1=W_2$ to be any disc bundle over $\Sigma$ that admits
a Stein structure, with even (resp.\ odd) Euler number $e$ if $V$ is the
trivial (resp.\ non-trivial) $S^2$-bundle.
(The condition for the existence of a
Stein structure is that the sum of the Euler number $e$ and the Euler
characteristic of $\Sigma$ be non-positive;
see~\cite[Exercise~11.2.5]{gost99}.)

\vspace{1mm}

(2) Now let $V$ be the complex projective plane $\cpk$.
Our aim is to show that the decomposition $\cpk = W_1\cup\oW_2$
can be chosen in such a way that we can realise every homotopy class
of almost contact structures on $\cpk\times S^1$ by a contact
structure.

Almost contact structures on $N:=\cpk\times S^1$ are determined
by their first Chern class, since $H^2(N)$ has no $2$-torsion.
(In \cite[Proposition~8.1.1]{geig08}
this assumption on $2$-torsion was erroneously omitted; for a correct
proof of the preceding statement see~\cite{hami08}.) The map
$[z_0:z_1:z_2]\mapsto [\oz_0:\oz_1:\oz_2]$ defines an orientation-preserving
diffeomorphism of $\cpk$ that acts as minus the identity on
$H^2(\cpk )$. Therefore, since the second Stiefel--Whitney class
$w_2(N)$ equals $1\in\Z_2= H^2(N;\Z_2)$, it suffices
to show that for any odd $c\in\N$ we can realise one of
$\pm c\in\Z\cong H^2(N)$
as the first Chern class of some contact structure; in other words,
we need not worry about signs.

Choose a smooth complex projective curve $C_d\subset\cpk$
of degree $d\geq 2$. Define $W_1$ to be the complement of
an open tubular neighbourhood of~$C_d$, with symplectic form given by
the restriction of the standard K\"ahler form on $\cpk$ to~$W_1$
(possibly adjusted as in the proof of Proposition~\ref{prop:exact}).

The genus $g=g(d)$ of $C_d$ --- determined by the adjunction
equality --- equals $g=(d-1)(d-2)/2$. The manifold $W_2$ in our general
construction now is the normal disc bundle of $C_d$, but with
reversed orientation. So the Euler number of this bundle
is $e=-[C_d]^2=-d^2$. (The condition for this to admit
a Stein structure is $0\geq 2-2g+e=3d-2d^2$, which is why we have to
rule out the case $d=1$.)

\begin{prop}
\label{prop:cp2}
The described decomposition of $\cpk$ (for any $d\geq 2)$ leads to a
contact structure on $\cpk\times S^1$ with first Chern class $\pm (2d-3)$.
This means that every homotopy class of almost contact structures on
$\cpk\times S^1$ can be realised by a contact structure.
\end{prop}

\begin{proof}
We begin with the case $d\geq 4$. Here $g>1$, so we can use the Stein
surface $W_2$ described in Figure~\ref{figure:stein} of
Section~\ref{section:decomp}. That is, we
attach a $2$-handle along the $(2g-2-e)$-fold positive or negative
stabilisation $S^{2g-2-e}_{\pm}K$ of the Legendrian knot $K$ shown in that
figure. In the present situation we have
\[ 2g-2-e= (d-1)(d-2)-2+d^2=2d^2-3d.\]
Recall from \cite[Section~11.3]{gost99} that the generator $[C_d]$ of
$H_2(W_2)$ corresponds to the $2$-handle
attached along $S^{2d^2-3d}_{\pm}K$ (with framing $-1$ relative to the
contact framing). Moreover, the first Chern class $c_1(W_2)$
evaluates on that generator as the rotation number
$\rot (S^{2d^2-3d}_{\pm}K)=\pm 2d^2-3d$ of the attaching circle.

The second homology of $\cpk$ (and of $\cpk\times S^1$) is generated by
the class $[\cpl ]=[C_d]/d$. Write $\eta=\ker\alpha$ for the contact
structure on $\cpk\times S^1$ obtained via the construction in the
proof of Theorem~\ref{thm:product}. Since $d\alpha =d\lambda_2$
over $\Int (\oW_2)\times S^1$, we can now compute (possibly up to signs)
\begin{eqnarray*}
\langle c_1(\eta ),[\cpl ]\rangle & = & \langle c_1(\eta ),[C_d]\rangle /d
      \;\; = \;\; \langle c_1(W_2),[C_d]\rangle /d \\
      & = & \rot (S^{2d^2-3d}_{\pm}K)/d
      \;\; = \;\; \pm (2d-3).
\end{eqnarray*}

In the case $d=3$ we have $g=1$ and $e=-9$. Lemma~3.9 of~\cite{hond00b}
still applies to show that the horizontal contact structure $\xi$
we want to realise on $M=\partial W_2$ is universally tight.
Since $(M,\xi )$ is symplectically filled by~$W_1$, it follows from
a result of Gay~\cite[Corollary~3]{gay06} that its torsion
in the sense of Giroux~\cite{giro00} is zero. This characterises
$\xi$ up to diffeomorphism, see~\cite[p.~686]{giro00}. Therefore,
we may again use the Stein manifold $W_2$ from Figure~\ref{figure:stein}
(with $g=1$), and the calculation of the first Chern class
of $\eta$ goes through as in the case~$d\geq 4$.

Finally, for $d=2$ we have $g=0$ and $e=-4$. So here the separating
hypersurface $M$ is the lens space $L(4,1)$ with a contact structure $\xi$
that admits a symplectic (in fact, Stein) filling by~$W_1$, whose
interior is diffeomorphic to the complement of a quadric in~$\cpk$;
the manifold $W_2$ is the disc bundle over $S^2$ with Euler class~$-4$.
From the Mayer--Vietoris sequence of the splitting $\cpk =W_1\cup_M\oW_2$
one finds $H_1(W_1)=\Z_2$. Moreover, by the Seifert--van~Kampen theorem,
the inclusion $L(4,1)\subset W_1$ must induce a surjective
homomorphism on fundamental groups. It follows that $\pi_1(W_1)=\Z_2$.

When we pass to the double cover of $W_1$, we obtain a symplectic filling
of the double cover $(\rpk ,\tilde{\xi})$ of $(L(4,1),\xi )$. Since
there is a unique tight contact structure on $\rpk$, {\em viz.}\ the one
covered by the standard contact structure on~$S^3$, we conclude that
$\xi$ is a universally tight contact structure on $L(4,1)$.

From the classification of tight contact structures on $L(4,1)$,
see~\cite{giro00} and \cite{hond00a}, we know that
there are two (resp.\ one) such structures up to isotopy
(resp.\ diffeomorphism), given by the surgery diagram in
Figure~\ref{figure:lens} and its mirror image. (Here we appeal again
to \cite[p.~435]{gost99} in order to see that only the
stabilisations $S^2_{\pm}K_0$ of the Legendrian unknot $K_0$ give
rise to a universally tight contact structure on $L(4,1)$,
whereas contact $(-1)$-surgery on $S_+S_-K_0$ gives the unique tight but
virtually overtwisted contact structure.)

\begin{figure}[h]
\labellist
\small\hair 2pt
\pinlabel $-1$ at 213 174
\endlabellist
\centering
\includegraphics[scale=0.4]{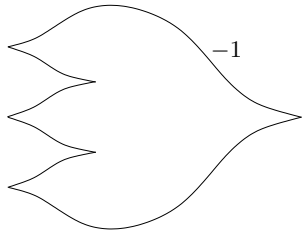}
  \caption{The universally tight contact structure on $L(4,1)$.}
  \label{figure:lens}
\end{figure}

Read as a Kirby diagram, that figure describes the desired
Stein filling $W_2$ of $(L(4,1),\xi )$ by the disc bundle
over $S^2$ with Euler class~$-4$. Finally, we have
$\rot (S^2_{\pm}K_0)/2=\pm 1$, as desired.
\end{proof}
\section{The $S^1$-connected sum}
\label{section:sum}
Let $\iota_i\co S^1\times D^4\rightarrow M_i$ ($i=1,2$) be two 
orientation-preserving embeddings
into oriented $5$-dimensional manifolds. Choose an orientation-reversing
diffeomorphism $\phi$ of $S^3=\partial D^4$. Then the
{\em $S^1$-connected sum\/} of the two manifolds $M_1,M_2$ is defined as
\[ M_1\#_{S^1}M_2:=\bigl( M_1\setminus\iota_1(S^1\times\Int (D^4)\bigr)
\cup_{\partial}
\bigl( M_2\setminus\iota_2(S^1\times\Int (D^4)\bigr) ,\]
where the boundaries are glued by
\[ \iota_1(\theta ,p)\sim\iota_2(\theta ,\phi (p))\;\;
\mbox{\rm for}\;\; \theta\in S^1,\; p\in S^3.\]

\begin{rem}
Given an embedding of $S^1\equiv S^1\times\{ 0\}$ into an oriented
$5$-manifold, there are two possible extensions (up to isotopy)
to an embedding of $S^1\times D^4$, since normal framings are classified by
$\pi_1(\SO (4))=\Z_2$. So the notation $\#_{S^1}$ is slightly ambiguous.
In many cases discussed in~\cite{hasu09} the diffeomorphism type
of the resulting manifold does not actually depend on the choice of
framing of the embedded circles. On the other hand, by Cerf's theorem there
is a unique orientation-reversing diffeomorphism of $S^3$ up to
isotopy, so the particular choice of $\phi$ is irrelevant. The construction
in the following theorem goes through in all (odd) dimensions,
but only with the specific choice of gluing described in the proof.
\end{rem}

\begin{thm}
\label{thm:sum}
Let $(M_1,\xi_1)$ and $(M_2,\xi_2)$ be two $5$-dimensional contact
manifolds. Given two orientation-preserving embeddings
$S^1\times D^4\rightarrow M_i$ ($i=1,2$), the corresponding
$S^1$-connected sum $M_1\#_{S^1}M_2$ carries a contact structure.
The embeddings of $S^1\times D^4$ can be isotoped in such a way
that the contact structure on $M_1\#_{S^1}M_2$ coincides
with $\xi_i$ on $M_i\setminus S^1\times D^4$.
\end{thm}

\begin{proof}
For the time being, we drop the subscript $i$ and consider only
a single embedding as in the theorem. The restriction of $\xi$ to
the embedded $S^1\equiv S^1\times\{ 0\}\subset M$ is a trivial
$\U (2)$-bundle. Therefore, by the $h$-principle for
isotropic immersions~\cite[Chapter~16]{elmi02}, we may assume that
$S^1$ is tangent to $\xi$.

The conformal symplectic normal bundle (see~\cite{geig08}) of
an isotropic $S^1$ in a $5$-dimensional contact manifold is a
trivial $\U (1)$-bundle. Since the natural homomorphism $\pi_1 (\U (1))
\rightarrow \pi_1 (\SO (4))$ is surjective, the trivialisation of
this conformal symplectic normal bundle can be chosen in such a way
that the corresponding framing of $S^1$ coincides with the one
given by the embedding of $S^1\times D^4$ into~$M$.

The neighbourhood theorem for isotropic
embeddings~\cite[Theorem~2.5.8]{geig08} then implies that in a
neighbourhood of $S^1$ the contact structure $\xi$ can be given by
the contact form
\[ \alpha:= x_1\, dy_1-y_1\, dx_1 + x_2\, dy_2-y_2\, dx_2 +
x_3\, dy_3-y_3\, dx_3,\]
where the normal framing of $S^1$ given by $\partial_{y_1},\partial_{y_2},
\partial_{x_3},\partial_{y_3}$ corresponds to that of the
given embedding. Here $S^1$ is being identified with
\[ S^1=\{ x_1^2+x_2^2=1,\; y_1=y_2=x_3=y_3=0\} ;\]
the conformal symplectic normal bundle is the trivial bundle
spanned by $\partial_{x_3},\partial_{y_3}$.

The vector field
\[ X:= y_1\partial_{y_1}+y_2\partial_{y_2}+\frac{1}{2}x_3\partial_{x_3}+
\frac{1}{2}y_3\partial_{y_3}\]
is a contact vector field transverse to the boundary $\Sigma :=S^1\times S^3$
of a small tubular neighbourhood of $S^1\subset M$, i.e.\
the flow of $X$ preserves~$\xi$. (So $\Sigma$ is a convex hypersurface
in the sense of Giroux~\cite{giro91}.) Our goal will be to
use this information in order to define a contactomorphism
of a neighbourhood of $\Sigma$ that sends $\Sigma$ to itself but reverses
the normal direction. Such a contactomorphism can then be used to
effect the desired gluing in the theorem.

Use the flow of $X$ to identify a neighbourhood of $\Sigma$ in $M$
with a neighbourhood of $\Sigma\equiv\Sigma\times\{ 0\}$ in
$\Sigma\times\R$; the vector field $X$ is then identified with~$\partial_t$,
where $t$ denotes the $\R$-coordinate. Write
\[ (x_1,x_2)=(\cos\theta ,\sin\theta )\]
and
\[ (y_1,y_2,x_3,y_3)=(\e^tu_1,\e^tu_2,\e^{t/2}v_3,\e^{t/2}w_3).\]
Thus, $\theta$ is the coordinate on the $S^1$-factor, and
$(u_1,u_2,v_3,w_3)$ may be interpreted as coordinates
on the $S^3$-factor of~$\Sigma$.

In these coordinates, the $X$-invariant contact form
$\alpha_0:=\e^{-t}\alpha$ is given by
\begin{eqnarray*}
\alpha_0  & =  & (u_1\cos\theta+u_2\sin\theta )\, dt
                  + (u_1\sin\theta -u_2\cos\theta )\, d\theta\\
          &    & \mbox{}+\cos\theta\, du_1+\sin\theta\, du_2
                 +v_3\, dw_3-w_3\, dv_3.
\end{eqnarray*}
Replace the coordinates $(u_1,u_2)$ by
\[ (v_1,v_2):=(u_1\cos\theta +u_2\sin\theta ,u_1\sin\theta -u_2\cos\theta ).\]
Then $\alpha_0$ is written as
\[ \alpha_0=dv_1+v_1\, dt +2v_2\, d\theta +v_3\, dw_3-w_3\, dv_3.\]
As explained in \cite[p.~180]{geig08}, we may in fact identify a
neighbourhood of $\Sigma$ in $M$ contactomorphically
with {\em all of\/} $\Sigma\times\R$ with this $\R$-invariant
contact form.

Now consider the following family of $\R$-invariant $1$-forms on $\R
\times\Sigma$:
\begin{eqnarray*}
\alpha_{\varphi} & := & dv_1+ (v_1\cos\varphi -v_2\sin\varphi )\, dt\\
                 &    & \mbox{}+2(v_1\sin\varphi+v_2\cos\varphi )\, d\theta
                        +v_3\, dw_3-w_3\, dv_3,\;\;\varphi\in [0,\pi /2].
\end{eqnarray*}
This is a contact form for each $\varphi\in [0,\pi /2]$ since,
on the universal cover $\R\times\R\times S^3$ of $\R\times\Sigma$,
this is simply the pull-back of $\alpha_0$ under the coordinate
transformation
\[ (t,\theta )\longmapsto (t\cos\varphi+2\theta\sin\varphi ,
                          -\frac{t}{2}\sin\varphi +\theta\cos\varphi ).\]
Therefore, Gray stability (which applies on the
non-compact manifold $\R\times\Sigma$ because of the $\R$-invariance of
our contact forms) gives us an isotopy whose time-$(\pi /2)$
map $\phi$ pulls back $\alpha_0$ to
\[ \alpha_{\pi/2}=dv_1-v_2\, dt+2v_1\, d\theta+v_3\, dw_3
-w_3\, dv_3,\]
up to multiplication by some positive function.

The map
\[ \psi\co (t,\theta ,v_1,v_2,v_3,w_3)\longmapsto
(-t,\theta ,v_1,-v_2,v_3,w_3)\]
is a contactomorphism for $\ker\alpha_{\pi /2}$ that sends $\Sigma$
to itself, reversing both the orientation of $\Sigma$
and its normal orientation. So $\phi\circ\psi\circ\phi^{-1}$
is a contactomorphism for $\xi =\ker\alpha_0$ that preserves
the isotopic copy $\phi (\Sigma )$ of $\Sigma$ while reversing its normal
orientation. As explained, such a contactomorphism allows us to
perform the $S^1$-connected sum of two contact manifolds.
\end{proof}
\section{Five-manifolds with fundamental group of order two}
\label{section:Z2}
In \cite{geth98} it was shown that every closed, orientable
$5$-manifold $M$ with fundamental group $\Z_2$ and
second Stiefel--Whitney class $w_2(M)$
equal to zero on homology admits a contact structure.
The basis of this result was a structure theorem, also proved
in~\cite{geth98}, according to which any such manifold can be obtained
from one of ten explicit model manifolds by surgery along $2$-spheres.

An orientable $5$-manifold is said to be of {\em fibred type}
if the second homotopy group $\pi_2(M)$ is a trivial $\Z [\pi_1(M)]$-module.
In \cite{hasu09} Hambleton and Su give an explicit description
of the fibred type $5$-manifolds with fundamental group $\Z_2$
and torsion-free second homology. Not all the manifolds
discussed in \cite{geth98} are of fibred type: one of the ten model manifolds
fails to be so, and surgery along $2$-spheres will also destroy
that property, in general. On the other hand, the list
in \cite{hasu09} contains manifolds where $w_2(M)$ does not vanish
on homology.

As Hambleton and Su point out, all the manifolds in their list
have vanishing third integral Stiefel--Whitney class and thus admit
an almost contact structure.

\begin{prop}
\label{prop:five}
Every closed, orientable, fibred type $5$-manifold with
fundamental group $\Z_2$ and torsion-free second homology
admits a contact structure.
\end{prop}

\begin{proof}
According to \cite[Theorem~3.6]{hasu09}, every such manifold
is an $S^1$-connected sum of some of the following manifolds:
\begin{itemize}
\item[(i)] one of the nine model manifolds from \cite{geth98}
with $w_2\neq 0$,
\item[(ii)] $S^2\times\rpk$,
\item[(iii)] $(\#_k S^2\times S^2)\times S^1$,
\item[(iv)] $\cpk\times S^1$.
\end{itemize}

By Theorem~\ref{thm:sum} it suffices to show that these individual
manifolds carry a contact structure. (There is no orientation issue,
because all the manifolds in this list admit an orientation-reversing
diffeomorphism.) The manifolds in (i) are covered by~\cite{geth98}.
The manifold in (ii) is the unit cotangent bundle of~$\rpk$.
The manifolds in (iii) and (iv) are covered by Corollary~\ref{cor:five}.
(For (iii) one may alternatively think of $S^2\times\ S^2\times S^1$
as the unit cotangent bundle of $S^2\times S^1$ and then take
$S^1$-connected sums.) This completes the proof.
\end{proof}

\begin{ack}
This project was initiated while H.~G.\ was a guest of the
Forschungsinstitut f\"ur Mathematik (FIM) of the ETH Z\"urich.
H.~G.\ thanks Dietmar Salamon for his hospitality, and the FIM for
its support. Most of the research was done during a stay of A.~S.\
at the Mathematical Institute of the Universit\"at zu K\"oln,
supported by the DFG Graduiertenkolleg ``Globale Strukturen in
Geometrie und Analysis''. A.~S.\ thanks the Mathematical Institute
for providing a stimulating working environment. We thank Su Yang for useful
correspondence, and Selman Akbulut for drawing our attention to the
work of Baykur. Last, but not least, we are grateful
to Paolo Lisca for allowing us to use his Figure~\ref{figure:stein}.
\end{ack}


\begin{thebibliography}{99}
%
\bibitem{akma98}
{\sc S. Akbulut and R. Matveyev},
A convex decomposition theorem for $4$-manifolds,
{\it Int. Math. Res. Not.}
{\bf 1998}, no.~7, 371--381.
%
\bibitem{bayk06}
{\sc R. \.{I}. Baykur},
K\"ahler decompositions of $4$-manifolds,
{\it Algebr. Geom. Topol.}
{\bf 6} (2006), 1239--1265.
%
\bibitem{bool97}
{\sc F. A. Bogomolov and B. de Oliveira},
Stein small deformations of strictly pseudoconvex surfaces,
in: {\it Birational Algebraic Geometry\/} (Baltimore, 1996),
Contemp. Math. {\bf 207},
American Mathematical Society, Providence, RI (1997), 25--41.
%
\bibitem{bowa58}
{\sc W. M. Boothby and H. C. Wang},
On contact manifolds,
{\it Ann. of Math. (2)\/}
{\bf 68} (1958), 721--734.
%
\bibitem{botu82}
{\sc R. Bott and L. W. Tu},
{\it Differential Forms in Algebraic Topology},
Grad. Texts in Math. {\bf 82},
Springer-Verlag, Berlin (1982).
%
\bibitem{bour02}
{\sc F. Bourgeois},
Odd dimensional tori are contact manifolds,
{\it Int. Math. Res. Not.}
{\bf 2002}, no.~30, 1571--1574.
%
\bibitem{dona96}
{\sc S. K. Donaldson},
Symplectic submanifolds and almost-complex geometry,
{\it J. Differential Geom.}
{\bf44} (1996), 666--705.
%
\bibitem{elia91}
{\sc Ya. Eliashberg},
On symplectic manifolds with some contact properties,
{\it J. Differential Geom.}
{\bf 33} (1991), 233--238.
%
\bibitem{elmi02}
{\sc Ya. Eliashberg and N. Mishachev},
{\it Introduction to the $h$-Principle},
Grad. Stud. Math. {\bf 48},
American Mathematical Society, Providence, RI (2002).
%
\bibitem{etny98}
{\sc J. B. Etnyre},
Symplectic convexity in low-dimensional topology,
{\it Topology Appl.}
{\bf88} (1998), 3--25.
%
\bibitem{gay06}
{\sc D. T. Gay},
Four-dimensional symplectic cobordisms containing three-handles,
{\it Geom. Topol.}
{\bf 10} (2006), 1749--1759.
%
\bibitem{geig97}
{\sc H. Geiges},
Constructions of contact manifolds,
{\it Math. Proc. Cambridge Philos. Soc.}
{\bf 121} (1997), 455--464.
%
\bibitem{geig97a}
{\sc H. Geiges},
Applications of contact surgery,
{\it Topology\/}
{\bf 36} (1997), 1193--1220.
%
\bibitem{geig06}
{\sc H. Geiges},
Contact Dehn surgery, symplectic fillings, and property P for knots,
{\it Expo. Math.}
{\bf 24} (2006), 273--280.
%
\bibitem{geig08}
{\sc H. Geiges},
{\it An Introduction to Contact Topology},
Cambridge Stud. Adv. Math. {\bf 109},
Cambridge University Press, Cambridge (2008).
%
\bibitem{geth98}
{\sc H. Geiges and C. B. Thomas},
Contact topology and the structure of $5$-manifolds with $\pi_1=\Z_2$,
{\it Ann. Inst. Fourier (Grenoble)\/}
{\bf 48} (1998), 1167--1188.
%
\bibitem{giro91}
{\sc E. Giroux},
Convexit\'e en topologie de contact,
{\it Comment. Math. Helv.}
{\bf 66} (1991), 637--677.
%
\bibitem{giro00}
{\sc E. Giroux},
Structures de contact en dimension trois et bifurcations des feuilletages
de surfaces,
{\it Invent. Math.}
{\bf 141} (2000), 615--689.
%
\bibitem{giro02}
{\sc E. Giroux},
G\'eom\'etrie de contact: de la dimension trois vers les
dimensions sup\'erieures,
in: {\it Proceedings of the International Congress of Mathematicians},
Vol.~2 (Beijing, 2002),
Higher Ed. Press, Beijing (2002), 405--414.
%
\bibitem{gost99}
{\sc R. E. Gompf and A. I. Stipsicz},
{\it $4$-Manifolds and Kirby Calculus},
Grad. Stud. in Math. {\bf 20},
American Mathematical Society, Providence, RI (1999).
%
\bibitem{hasu09}
{\sc I. Hambleton and Y. Su},
On certain $5$-manifolds with fundamental group of order~$2$,
preprint, {\tt arXiv:0903.5244v3} (2009).
%
\bibitem{hami08}
{\sc M. Hamilton},
On symplectic $4$-manifolds and contact $5$-manifolds,
Dissertation, Ludwig-Maximilians-Universit\"at, M\"unchen (2008).
%
\bibitem{hiho58}
{\sc F. Hirzebruch and H. Hopf},
Felder von Fl\"achenelementen in $4$-dimensionalen Mannigfaltigkeiten,
{\it Math. Ann.}
{\bf 136} (1958), 156--172.
%
\bibitem{hond00a}
{\sc K. Honda},
On the classification of tight contact structures~I,
{\it Geom. Topol.}
{\bf 4} (2000), 309--368;
erratum:
Factoring nonrotative $T^2\times I$ layers,
{\it Geom. Topol.}
{\bf 5} (2001), 925--938.
%
\bibitem{hond00b}
{\sc K. Honda},
On the classification of tight contact structures~II,
{\it J. Differential Geom.}
{\bf 55} (2000), 83--143.
%
\bibitem{lerm04}
{\sc E. Lerman},
Contact fiber bundles,
{\it J. Geom. Phys.}
{\bf 49} (2004), 52--66.
%
\bibitem{list04}
{\sc P. Lisca and A. I. Stipsicz},
Tight, not semi-fillable contact circle bundles,
{\it Math. Ann.}
{\bf 328} (2004), 285--298.
%
\bibitem{mcdu91}
{\sc D. McDuff},
Symplectic manifolds with contact type boundaries,
{\it Invent. Math.}
{\bf 103} (1991), 651--671.
%
\bibitem{mcsa98}
{\sc D. McDuff and D. Salamon},
{\it Introduction to Symplectic Topology},
2nd edition, Oxford Math. Monogr.,
Oxford University Press, Oxford (1998).
%
\bibitem{nied05}
{\sc K. Niederkr\"uger},
Compact Lie group actions on contact manifolds,
Dissertation, Universit\"at zu K\"oln (2005).
%
\bibitem{ozst04}
{\sc B. \"Ozba\u{g}c\i\ and A. I. Stipsicz},
{\it Surgery on Contact $3$-Manifolds and Stein Surfaces},
Bolyai Soc. Math. Stud. {\bf 13},
Springer-Verlag, Berlin (2004).
%
\bibitem{wein91}
{\sc A. Weinstein},
Contact surgery and symplectic handlebodies,
{\it Hokkaido Math. J.}
{\bf 20} (1991), 241--251.
\end{thebibliography}
\end{document}